\documentclass[11pt]{amsart}
\usepackage{amsmath}
\usepackage{amsthm}
\usepackage{amssymb}
\usepackage{amscd}

\theoremstyle{plain}
\newtheorem{theorem}{Theorem}
\newtheorem{lemma}[theorem]{Lemma}
\newtheorem{corollary}[theorem]{Corollary}
\newtheorem*{case1}{Case 1}
\newtheorem*{case2}{Case 2}
\newtheorem*{case3}{Case 3}

\newtheorem*{claim}{Claim}

\theoremstyle{definition}
\newtheorem{definition}[theorem]{Definition}
\newtheorem{example}[theorem]{Example}

\theoremstyle{remark}
\newtheorem*{remark}{Remark}

\DeclareMathOperator{\HOD}{HOD}
\DeclareMathOperator{\OD}{OD}
\DeclareMathOperator{\ZFC}{ZFC}
\DeclareMathOperator{\cof}{cof}
\DeclareMathOperator{\cf}{cf}

\DeclareMathOperator{\crit}{crit}

\begin{document}

\title{The HOD Dichotomy}

\author{W. Hugh Woodin}
\author{Jacob Davis}
\author{Daniel Rodr\'{i}guez}

\maketitle

\section{Introduction} \label{intro}

This paper provides a more accessible account of some of the material from Woodin
\cite{Woodin} and \cite{Woodin2}.  All unattributed results are due to the first author.

Recall that $0^\#$ is a certain set of natural numbers that codes
an elementary embedding $j : L \to L$ such that $j \not= \mathrm{id} \upharpoonright L$.
Jensen's covering lemma says that if $0^\#$ does not exist and $A$ is an uncountable set
of ordinals, then there exists $B \in L$ such that $A \subseteq B$ and $|A| = |B|$.
The conclusion implies that if $\gamma$ is a singular cardinal,
then it is a singular cardinal in $L$.
It also implies that if $\gamma \ge \omega_2$ and $\gamma$ is a successor cardinal in $L$,
then $\mathrm{cf}(\gamma) = |\gamma|$.
In particular,
if $\beta$ is a singular cardinal,
then $(\beta^+)^L = \beta^+$.
Intuitively, this says that $L$ is close to $V$.
On the other hand,
should $0^\#$ exist,
if $\gamma$ is an
uncountable cardinal,
then $\gamma$ is an inaccessible cardinal in $L$.
In this case,
we could say that
$L$ is far from $V$.
Thus,
the covering lemma has the following corollary,
which does not mention $0^\#$.

\begin{theorem}[Jensen]
Exactly one of the following holds.
\begin{enumerate}
\item $L$ is correct about singular cardinals
and computes their successors correctly.
\item Every uncountable cardinal is inaccessible in $L$.
\end{enumerate}
\end{theorem}

Imagine an alternative history in which this $L$ dichotomy was discovered
without knowledge of $0^\#$ or more powerful large cardinals.
Clearly, (1) is consistent because it holds in $L$.
On the other hand, 
whether or not there is a proper class of inaccessible cardinals in $L$
is absolute to generic extensions.
This incomplete evidence might have led set theorists to conjecture that (2) fails.
Of course,
(2) only holds when $0^\#$ exists but
$0^\#$ does not belong to $L$ and $0^\#$ cannot be added by forcing.

Canonical inner models other than $L$ have been defined and shown to satisfy 
similar covering properties and corresponding dichotomies.
Part of what makes them canonical is that they are contained in HOD.
In these notes, we will prove a dichotomy theorem of this kind for HOD itself.
Towards the formal statement,
recall that a cardinal $\delta$ is {\em extendible} iff for every $\eta > \delta$,
there exists $\theta > \eta$ and an elementary embedding $j : V_{\eta + 1} \to V_{\theta +1}$
such that $\mathrm{crit}(j) = \delta$ and $j(\delta) > \eta$.
The following result expresses the idea
that either HOD is close to $V$ or else HOD is far from $V$.
We will refer to it as the HOD Dichotomy.

\begin{theorem}\label{hoddichotomy}
Assume that $\delta$ is an extendible cardinal.
Then exactly one of the following holds.
\begin{enumerate}
\item For every singular cardinal $\gamma > \delta$,
$\gamma$ is singular in $\HOD$ and\\
$(\gamma^+)^{\HOD} = \gamma^+$.
\item Every regular cardinal greater than $\delta$ is measurable in $\mathrm{HOD}$.
\end{enumerate}
\end{theorem}

In this note, we shall prove a dichotomy in which (2) is weakened to
hold for all sufficiently large regular cardinals greater than $\delta$;
see Corollary \ref{final}.   The full result can be found in \cite{Woodin} Theorem 212.

Notice that we have stated the HOD dichotomy
without deriving it from a covering property that
involves a ``large cardinal missing from HOD''.
In other words, no analogue of $0^\#$ is mentioned
and the alternative history we described for $L$
is what has actually happened in the case of $\HOD$.
This leads us to  conjecture that (2) fails.
One reason is that (2) is absolute between $V$ and its generic extensions
by posets that belong to $V_\delta$, which we will show this in the next section.
There is some evidence for this conjecture. All known large cardinal axioms (which do not contradict the Axiom of Choice) are compatible  with $V=\HOD$ and so trivially cannot imply (2). Further, we shall see that the main technique for obtaining independence in set theory (forcing) probably cannot be used to show that (2) is relatively consistent with the existence of an extendible cardinal starting from any know large cardinal hypothesis which is also consistent with the Axiom of Choice.  Finally, by definition $\HOD$ contains all definable sets of ordinals and this makes it difficult to imagine a meaningful analogue of $0^{\#}$ for $\HOD$.

Besides evidence in favor of this conjecture about HOD,
we also have applications.  Recall that Kunen proved in ZFC that there is no
non-trivial elementary embedding from $V$ to itself.  It is a longstanding
open question whether this is a theorem of ZF alone. One of our applications
is progress on this problem.  This and other applications will be listed in
Section \ref{consequences}.

\section{Generic absoluteness} \label{absolute}

In this section,
we establish some basic properties of forcing and HOD,
and use them to show that the conjecture about HOD from the previous section
is absolute to generic extensions.
In other words,
if $\mathbb{P}$ is a poset,
then clause (2) of Theorem \ref{hoddichotomy}
holds in $V$ iff it holds in every generic extension by $\mathbb{P}$.

First observe that if $\mathbb{P}$ is a weakly homogeneous (see \cite{Jech} Theorem 26.12) and ordinal definable poset in $V$,
and $G$ is a $V$-generic filter on $\mathbb{P}$,
then $\mathrm{HOD}^{V[G]} \subseteq \mathrm{HOD}^V$.
This is immediate from the basic fact about weakly homogeneous forcing
that for all $x_1 , \dots , x_n \in V$ and formula $\varphi (v_1, \dots , v_n)$,
every condition in $\mathbb{P}$ decides $\varphi(\check{x}_1, \dots, \check{x}_n)$
the same way. We also use here that a class model of ZFC can be identified solely from its sets of ordinals, since each level of its $V$ hierarchy can, using the Axiom of Choice, be encoded by a relation on $|V_\alpha|$ and then recovered by collapsing. We shall use this fact repeatedly.

Let us pause to give an example of the phenomenon we just mentioned
in which HOD of the generic extension is properly contained in HOD of the ground model.
Let $\mathbb{P}$ be Cohen forcing and $g : \omega \to \omega$ be a Cohen real over $L$.
Of course, $g \not\in L$.
In $L[g]$,
let $\mathbb{Q}$ be the Easton poset that forces
\begin{equation*}
2^{\omega_n}=\begin{cases} \omega_{n+1} \quad g(n)=0\\
	\omega_{n+2} \quad g(n)=1.\\
	\end{cases}
\end{equation*}
Both $\mathbb{P}$ and $\mathbb{Q}$ are cardinal preserving.
Now let $H$ be an $L[g]$-generic filter on $\mathbb{Q}$.
Observe that 
$g \in \mathrm{HOD}^{L[g][H]}$ because it can
be read off from $\kappa \mapsto 2^\kappa$ in $L[g][H]$.
Now let $\lambda$ be a regular cardinal greater than $|\mathbb{P} * \mathbb{Q}|$.
Then $\mathbb{P} * \mathbb{Q} * \mathrm{Coll}(\omega,\lambda)$
and $\mathrm{Coll}(\omega,\lambda)$ have isomorphic Boolean completions,
so there is an $L$-generic filter $J$ on $\mathrm{Coll}(\omega,\lambda)$
and an $L[g][H]$-generic filter $I$ on $\mathrm{Coll}(\omega,\lambda)$
such that
$L[J] = L[g][H][I]$.
Using the fact that
$\mathrm{Coll}(\omega,\lambda)$ is definable and weakly homogeneous
we see that
$$L = \mathrm{HOD}^{L[J]} = \mathrm{HOD}^{L[g][H][I]} 
\subsetneqq \mathrm{HOD}^{L[g][H]}$$
where the inequality is witnessed by the Cohen real $g$.

An important fact about forcing which was discovered relatively recently
is that if  $\delta$ is a regular uncountable cardinal and $\mathbb{P} \in V_\delta$
is a poset,
then $V$ is definable from $\mathcal{P}(\delta) \cap V$ in $V[G]$.
Towards the precise statement and proof,
we make the following definitions.

\begin{definition}
Let $\delta$ be a regular uncountable cardinal and $N$ be a transitive class model of ZFC.
Then
\begin{itemize}
\item 
$N$ has the {\em $\delta$-covering property} iff for every $\sigma \subseteq N$ 
with $|\sigma| <\delta$, there exists  $\tau \in N$ such that $|\tau| <\delta$ and
$\tau \supseteq \sigma$, and
\item
$N$ has the {\em $\delta$-approximation property} iff 
for every cardinal $\kappa$ with $\mathrm{cf}(\kappa) \geq \delta$ and every 
$\subseteq$-increasing sequence of sets
$\langle \tau_\alpha \mid \alpha < \kappa \rangle$ from $N$, $\bigcup \tau_\alpha \in N$.
\end{itemize}
\end{definition}

By Jensen's theorem,
$L$ has the $\delta$-covering property in $V$ for every regular $\delta > \omega$
if $0^\#$ does not exist.
Next,
we show that $V$ has covering and approximation properties
in its generic extensions.

\begin{lemma}\label{dcda}
Let $\delta > \omega$ regular and $\mathbb{P}$ a poset with $|\mathbb{P}|<\delta$. 
Then $V$ has $\delta$-covering and $\delta$-approximation in $V[G]$ whenever
$G$ is a $V$-generic filter on $\mathbb{P}$.
\end{lemma}

\begin{proof}
First, we show the covering property.
Let $\sigma$ be a name such that $\Vdash \sigma \subset V \mbox{ and } |\sigma| < \delta$.
By the $\delta$ chain condition,
there are fewer than $\delta$ possible values of $|\sigma|$.
Let $\gamma < \delta$ be the supremum of these and pick $\dot f$
such that $\Vdash \dot f : \gamma \twoheadrightarrow \sigma$.
To finish this part of the proof,
let $\tau$ be the set of possible values for $\dot f ( \alpha ) $ and $\alpha < \gamma$.

Second, we prove the approximation property.
Say $p$ forces that $\mathrm{cf}(\kappa) \geq \delta$
and $\langle \tau_\alpha \mid \alpha < \kappa \rangle$ 
is an increasing sequence of sets from $V$.
For $\alpha < \kappa$,
let $p_\alpha$ decide the value of $\tau_\alpha$.
Because $|\mathbb{P}| < \delta \leq \cf(\kappa) \le \kappa$ there must be some 
$p_\beta$ that is repeated cofinally often and so determines $\bigcup \tau_{\alpha}$,
thereby forcing the union to belong to $V$.
By density,
the union is forced to belong to $V$.
\end{proof}

The next theorem is the promised result on the definability of the ground model,
which we state somewhat more generally.
Part (1) is due to Hamkins and (2) to Laver and Woodin independently.

\begin{theorem}\label{hlw}
Let $\delta$ be a regular uncountable cardinal.
Suppose that $M$ and $N$ are
transitive class model of ZFC that satisfies the 
$\delta$-covering and $\delta$-approximation properties, $\delta^+=(\delta^+)^N=(\delta^+)^M$,
and
$$N \cap \mathcal{P}(\delta) = M \cap \mathcal{P}(\delta).$$
(1) Then $M = N$.

\noindent (2) In particular,
$N$ is $\Sigma_2$-definable from $N \cap \mathcal{P}(\delta)$. 
\end{theorem}

\begin{proof}
For part (1) we show by recursion on ordinals $\gamma$ that for all $A \subseteq \gamma$
$$A \in M \iff A \in N.$$
The case $\gamma \le \delta$ is clear.
By the induction hypothesis,
$M$ and $N$ have the same cardinals $\le \gamma$,
and, if $\gamma$ is not a cardinal in these models,
then they have the same power set of $\gamma$.
Thus, we may assume that $\gamma$ is a cardinal of both $M$ and $N$.

\begin{case1}
$\mathrm{cf}(\gamma) \ge \delta$
\end{case1}

Then, $A \in M$ iff $A \cap \alpha \in M$ for
every $\alpha < \gamma$.  The forward direction is clear.  For the
reverse, use the $\delta$-approximation property to see
$$A = \bigcup \{ A \cap \alpha \mid \alpha < \gamma \} \in M.$$
The same holds for $N$.

\begin{case2}
$\gamma > \delta$, $\mathrm{cf}(\gamma) < \delta$ and $|A| < \delta$
\end{case2}

Define increasing sequences $\langle E_\alpha \mid \alpha < \delta \rangle$
and $\langle F_\alpha \mid \alpha < \delta \rangle$
of subsets of $\gamma$ such that
$| E_\alpha |, | F_\alpha | < \delta$, 
$A \subseteq E_0$,
$E_\alpha \subseteq F_\alpha$,
$\bigcup_{\alpha < \beta} F_\alpha  \subseteq E_\beta$,
$E_\alpha \in M$ and $F_\alpha \in N$.
For the construction, use the $\delta$-covering property alternately for $M$ and $N$. 
Then define $E = \bigcup E_\alpha  = \bigcup F_\alpha$ and note that 
$E \in M \cap N$ by $\delta$-approximation property. 
Let $\theta$ be the order-type of $E$
and $\pi : E  \to \theta$ the Mostowski collapse.
Then $\pi \in M \cap N$.
Also,
$\theta < \delta^+ = (\delta^+)^M = (\delta^+)^N$ because $|E| \leq \delta$.
By the induction hypothesis,
$$A \in M \iff  \pi[A] \in M \iff  \pi[A] \in N \iff A \in N.$$

\begin{case3}
$\gamma > \delta$, $\mathrm{cf}(\gamma) < \delta$ and $|A| \ge \delta$
\end{case3}

We claim that $A \in M$ iff
	\begin{enumerate}
		\item[(i)$_M$] for every $\alpha < \gamma$, $A \cap \alpha \in M$ and
		\item[(ii)$_M$] for every $\sigma \subseteq \gamma$,
		if $|\sigma|<\delta$ and $\sigma \in M$, then  $A \cap \sigma \in M$.
	\end{enumerate}
We also claim that $A \in N$ iff (i)$_N$ and (ii)$_N$.
The induction hypothesis is that (i)$_M$ iff (i)$_N$
and in case (2) we showed that (ii)$_M$ iff (ii)$_N$,
so our claim  implies $A \in M$ iff $A \in N$ as desired.

The forward implication of the claim is obvious,
so assume (i)$_M$ and (ii)$_M$.
Pick $\theta$ with $\mathrm{cf}(\theta) > \gamma$ and the defining formula
for $M$ absolute to $V_\theta$.
Define an increasing chain $\langle X_\alpha \mid \alpha < \delta \rangle$
of elementary substructures of $V_\theta$
and an increasing chain $\langle Y_\alpha \mid \alpha < \delta \rangle$
of subsets of $V_\theta \cap M$
such that
$|X_{\alpha}|,|Y_{\alpha}|<\delta$,
$A \in X_0$,
$\sup (X_0 \cap \gamma) = \gamma$,
$X_\alpha \cap N \subseteq Y_\alpha$, $Y_\alpha \in M$ and
$\bigcup_{\alpha<\beta} (Y_\alpha \cup X_\alpha ) \subseteq X_\beta$.
We use Downward Lowenheim-Skolem to obtain $X_\alpha$
and the 
$\delta$-covering property to obtain $Y_{\alpha}$.
Define $X = \bigcup X_{\alpha}$ and $Y= \bigcup Y_{\alpha}$.  Then $X \prec V_\theta$
and $Y = X \cap M \prec V_\theta \cap M$.
By the $\delta$-approximation property,
$Y \in M$.
By (ii)$_M$,
for every $\alpha < \delta$, 
$A \cap Y_\alpha \in M$.
Again, by the $\delta$-approximation property, $A \cap Y \in M$.
Now consider an arbitrary $\alpha \in Y \cap \gamma$ and observe that
\begin{itemize}
\item $A \cap \alpha \in Y$ because $A \in X$ so $A\cap\alpha\in X$, and $A\cap\alpha \in M$ by (ii)$_M$; and 
\item for every $b \in Y$,
if $b \cap Y = (A \cap Y) \cap \alpha$, then $Y \models b = A \cap \alpha$, so $b = A \cap \alpha$.
\end{itemize}
Here we have used (i)$_M$ and $Y \prec V_\theta \cap M$.
So the sequence
$\langle A \cap \alpha \mid \alpha \in Y \cap \gamma \rangle$
is definable in $M$ from parameters $\gamma$, $Y$ and $A \cap Y$.
In particular, this function belongs to $M$.
The union of its range  is $A$, so $A \in M$.

Part (2) now follows. $A \in N$ iff there is a large regular $\theta$ and a model $M \subseteq V_{\theta}$ of ZFC-Power Set satisfying $\delta$-covering and $\delta$-approximation in $V_{\theta}$ such that $M \cap \mathcal{P}(\delta) = N \cap \mathcal{P}(\delta)$ and $A \in M$. This is a $\Sigma_2$ statement.
\end{proof}

We will use the following amazing result.
The final equality is not as well known, so we include a proof. 
Note that $\OD_{\mathcal{A}}$ here denotes the class of all sets that are definable using ordinals and members of $\mathcal{A}$, and $\HOD_\mathcal{A}$ is defined correspondingly.

\begin{theorem}[Vop\v{e}nka]\label{vopenka}
For every ordinal $\kappa$,
there exists $\mathbb{B} \in \mathrm{HOD}$
such that 
$$\mathrm{HOD} \models \mathbb{B} \mbox{ is a complete Boolean algebra}$$
and,
for every $a \subseteq \kappa$,
there exists a $\mathrm{HOD}$-generic filter $G$ on $\mathbb{B}$
such that 
$$\mathrm{HOD}[a] \subseteq \mathrm{HOD}_{\{G\}} = \mathrm{HOD}_{\{a\}} = \mathrm{HOD}[G].$$
\end{theorem}

\begin{proof}
First define $\mathbb{B}^*$ to be $\mathcal{P}(\mathcal{P} ( \kappa ) )\cap \mathrm{OD}$
with its Boolean algebra structure.
Then $\mathbb{B}^* \in \mathrm{OD}$ and 
$\mathbb{B}^*$ is $\mathrm{OD}$-complete.
Given $a \subseteq \kappa$,
we let $$G^* = \{ X \in \mathbb{B}^* \mid a \in X \}$$
and see that $G^*$ is an $\mathrm{OD}$-generic filter on $\mathbb{B}^*$.
Fix a definable 
bijection $\pi$ from 
$\mathcal{P}(\mathcal{P} ( \kappa )) \cap \mathrm{OD}$ to an ordinal.
Define $\mathbb{B}$ so that $\pi : \mathbb{B}^* \simeq \mathbb{B}$.
Let $G = \pi[G^*]$.
Then $G$ is a $\mathrm{HOD}$-generic filter on $\mathbb{B}$.
It is straightforward to see that $G \in \mathrm{HOD}_{\{ a \}}$
so it remains to see that $\mathrm{HOD}_{\{a\}} \subseteq \mathrm{HOD}[G]$.
Let $S \in \mathrm{HOD}_{\{a\}}$;
 we may assume $S$ is a set of ordinals. 
Say
$$S = \{ \zeta  < \theta  \mid V_\theta \models \varphi ( \zeta , \eta_1 , \dots , \eta_n , a ) \}.$$
For each $\zeta < \theta$,
let $$X_\zeta = \{ b \subseteq \kappa \mid 
V_\theta \models \varphi ( \zeta , \eta_1 , \dots , \eta_n , b ) \}.$$
Then $\zeta \mapsto \pi(X_\zeta)$ belongs to $\mathrm{HOD}$.
So $S = \{ \zeta \mid \pi (X_\zeta ) \in G \}$ belongs to $\mathrm{HOD}[G]$.\end{proof}

Combining the results in this section, we obtain the following.

\begin{corollary} \label{hodgen}
Let $\mathbb{P} \in \mathrm{OD}$ be a weakly homogeneous poset.
Suppose
G is a  $V$-generic filter on $\mathbb{P}$.
Then $\mathrm{HOD}^V$ is a generic extension of $\mathrm{HOD}^{V[G]}.$
\end{corollary}

\begin{proof}
Fix $\delta > |\mathbb{P}|$. 
By Lemma \ref{dcda} and Theorem \ref{hlw}, 
$V$ is definable in $V[G]$ from $A = \mathcal{P}(\delta)\cap V$.
In $V$, let $\kappa = |A|$ and $E$ be a binary relation on $\kappa$
such that the Mostowski collapse of $(\kappa,E)$ is $(\mathrm{trcl}(\{A\},\in)$.
Then 
$V_\gamma \in \mathrm{OD}^{V[G]}_{\{E\}}$ for every $\gamma$,
therefore 
$$\mathrm{HOD}^V \subseteq \mathrm{HOD}_{\{ E \}}^{V[G]}.$$
By Theorem \ref{vopenka},
we have a $\mathrm{HOD}^{V[G]}$-generic filter $H$ on a Vop\v{e}nka algebra
so that
$$\mathrm{HOD}_{\{ E \}}^{V[G]} = \mathrm{HOD}^{V[G]} [H].$$
Combining all of the above gives
$$\mathrm{HOD}^{V[G]} \subseteq \mathrm{HOD}^V \subseteq \mathrm{HOD}_{\{E\}}^{V[G]} = \mathrm{HOD}^{V[G]}[H].$$
As $\mathrm{HOD}^V$ is nested between $\mathrm{HOD}^{V[G]}$ and a generic extension thereof, it is itself a generic extension of $\mathrm{HOD}^{V[G]}$ (see \cite{Jech} Theorem 15.43).
\end{proof}

Finally, we discuss again our 
conjecture that clause (2) of Theorem \ref{hoddichotomy} fails.
Let us temporarily call this the HOD conjecture although a slightly different
statement will get this name later.
We wish to see that this conjecture is absolute between $V$ and its generic extensions.
Of course,
Theorem \ref{hoddichotomy} has an extendible cardinal $\delta$ in its hypothesis.
We should assume that we are forcing with a poset $\mathbb{P} \in V_\delta$
to assure  that if $G$ is a $V$-generic filter on $\mathbb{P}$,
then $\delta$ remains extendible in $V[G]$.

To see this, given $\eta>\delta$ limit observe that for each member of $V_\eta^{V[G]}$ we can, by induction on $\eta$ build a name in $V_\eta$ for that member. This is done as usual for nice names by considering maximal antichains, taking advantage of the fact $\mathbb{P}$ is small with respect to $\eta$. Thus $V_\eta[G]=(V_\eta)^{V[G]}$. Now take $j : V_\eta \to V_\theta$ elementary and define $\tilde{\jmath}:V_\eta^{V[G]}\to V_\theta^{V[G]}$ by $j(\tau_G) = j(\tau)_G$. This is a variation on the proof that measurability is preserved by small forcing; see \cite{Jech} Theorem 21.2 or \cite{LevySolovay} Theorem 3.

\begin{corollary}
The following statement
is absolute between $V$ and its generic extensions by posets in $V_\delta$: 
``$\delta$ is an extendible cardinal and for every singular cardinal $\gamma > \delta$,
$\gamma$ is singular in $\mathrm{HOD}$ and
$(\gamma^+)^\mathrm{HOD} = \gamma^+$.''
\end{corollary}

\begin{proof}
If $\mathbb{P}$ is ordinal definable and weakly homogeneous,
then it is clear from Corollary \ref{hodgen}
that the HOD conjecture is absolute between $V$ and $V[G]$.
Now consider the general case.
Take $\kappa < \delta$ an inaccessible cardinal such that
$\mathbb{P} \in V_\kappa$.
Let $J$ be a $V[G]$-generic filter on $\mathrm{Coll}(\omega,\kappa)$
and $I$ be a $V$-generic filter on $\mathrm{Coll}(\omega,\kappa)$
such that $V[G][I] = V[J]$.
Now $\mathrm{Coll}(\omega,\kappa)$ is ordinal definable and weakly homogeneous so the HOD conjecture is absolute between $V$ and $V[J]$,
as well as between $V[G]$ and $V[G][I]$.
Therefore,
it is absolute between $V$ and $V[G]$.
\end{proof}

\section{The HOD Conjecture} \label{hod}

The official HOD Conjecture is closely related to the conjecture we have been contemplating
for two sections.  Intuitively, it also says that HOD is not far from $V$,
which will turn out to mean that they are close.
The HOD Conjecture involves a new concept, which we define first.

\begin{definition}
Let $\lambda$ be an uncountable regular cardinal.
Then $\lambda$ is {\em $\omega$-strongly measurable in $\HOD$} iff
there is $\kappa < \lambda$ such that
\begin{enumerate}
\item $(2^\kappa)^{\HOD} <\lambda$ and
\item there is no partition $\langle S_\alpha \mid  \alpha <  \kappa \rangle$ 
of $\cof(\omega)\cap \lambda$ into stationary sets such that  
 $\langle S_\alpha \mid  \alpha <\kappa \rangle \in \HOD$.
\end{enumerate}
\end{definition}

\begin{lemma}\label{omega}
Assume $\lambda$ be $\omega$-strongly measurable in $\mathrm{HOD}$.
Then $$\mathrm{HOD} \models
\lambda \mbox{ is a measurable cardinal.}$$
\end{lemma}

\begin{proof}
We claim that there exists a stationary set $S \subseteq \mathrm{cof}(\omega) \cap \lambda$
such that $S \in \mathrm{HOD}$ and there is no partition of $S$
into two stationary sets that belong to $\mathrm{HOD}$.

First, let us see how to finish proving the lemma based on the claim.
Let $\mathcal{F}$ be the club filter restricted to $S$.  That is,
$$\mathcal{F} = 
\{ X \subseteq S \mid \mbox{ there is a club $C$ such that $X \supseteq C \cap S$}\}.$$
Let $\mathcal{G} = \mathcal{F} \cap \mathrm{HOD}$.
Clearly,
$\mathcal{G} \in \mathrm{HOD}$
and
$$\mathrm{HOD}
\models 
\mbox{$\mathcal{G}$ is a $\lambda$-complete filter on 
$\mathcal{P}(S)$.}$$
By the claim,
$$\mathrm{HOD}
\models 
\mbox{$\mathcal{G}$ is an ultrafilter on 
$\mathcal{P}(S)$.}$$

Now we prove the claim by contradiction.
Fix a cardinal
$\kappa < \lambda$ such that
$(2^\kappa)^\mathrm{HOD} <\lambda$ and
there is no partition
$\langle S_\alpha \mid  \alpha <  \kappa \rangle$ of $\cof(\omega)\cap \lambda$ 
into stationary sets such that  
$\langle S_\alpha \mid  \alpha < \kappa \rangle \in \mathrm{HOD}$.
This allows us to define
a subtree $T$ of ${}^{\le \kappa} 2$ with height $\kappa+1$
and a sequence
$\langle S_r \mid r \in T \rangle$ that belongs to  $\mathrm{HOD}$
such that
\begin{enumerate}
\item $S_{\langle \, \rangle} = \mathrm{cof}(\omega) \cap \lambda$,
\item For every $r \in T$,
\begin{enumerate}
\item $S_r$ is stationary,
\item $r {}^\frown \langle 0 \rangle$ and $r {}^\frown \langle 0 \rangle$ belong to $T$,
\item$S_r$ is the disjoint union of  $S_{r {}^\frown \langle 0 \rangle}$ and $S_{r {}^\frown \langle 1 \rangle}$, and
\item if $\mathrm{dom}(r)$ is a limit ordinal,
then $S_r = \bigcap \{ S_{r\upharpoonright \alpha} \mid \alpha \in \mathrm{dom}(r) \}$.
\end{enumerate}
\item For every limit ordinal $\beta \le \kappa$ and $r \in {}^\beta 2 - T$,
if $r \upharpoonright \alpha \in T$ for every $\alpha < \beta$,
then $\bigcap_{\alpha<\beta} S_{r\upharpoonright \alpha}$ is non-stationary.
\end{enumerate}
First notice that $\cof(\omega)\cap \lambda$ belongs to HOD
even though it might mean something else there.
Also, 
$\{ S \subseteq  \lambda \mid S \in \mathrm{HOD}
\mbox{ and $S$ is stationary}\}$ belongs to $\mathrm{HOD}$
even through there might be sets which are stationary in  HOD but not actually stationary.
In any case, HOD can recognise when a given $S \in \mathrm{HOD}$
is stationary in $V$ and,
by the putative failure of the claim,
choose a partition of $S$ into two sets which are again 
stationary in $V$.
This choice is done in a uniform way
using a wellordering of 
$$\{ S \subseteq  \lambda \mid S \in \mathrm{HOD} \mbox{ and $S$ is stationary}\}$$ 
in HOD.
This gets us through successor stages of the construction.
Suppose that $\beta \le \kappa$ is a limit ordinal and that
we have already constructed in $\HOD$
$\langle S_r \mid r \in T \cap {}^{<\beta} 2 \rangle$.
By (3) we have recursively maintained that,
except for a non-stationary set,
$\cof(\omega)\cap \lambda$ equals
$$\bigcup \left\{ \bigcap \left\{ S_{r\upharpoonright\alpha}
\mid \alpha < \beta \right\} \mid  
\mbox{$r$  is a $\beta$-branch of $T\cap {}^{<\beta} 2$ and $r \in \mathrm{HOD}$} \right\}.$$
Since the club filter over $\lambda$ is $\lambda$-complete and $| {}^\beta 2|^\mathrm{HOD} < \lambda$,
there exists at least one such $r$ for which
the corresponding intersection is stationary.
We put $r \in T \cap {}^\beta 2$ and define
$S_r = \bigcap \left\{ S_{r\upharpoonright\alpha}
\mid \alpha < \beta \right\}$  in this case.
That completes the construction.
Now take any $r \in T$ with $\mathrm{dom}(r) = \kappa$.
Then
$S_r$ is the disjoint union of the stationary sets
$S_{r\upharpoonright(\alpha + 1)} - S_{r\upharpoonright\alpha}$
for $\alpha < \kappa$.
This readily contradicts our choice of $\kappa$.\end{proof}

\begin{definition}\label{HODConjecture}
The {\em $\mathrm{HOD}$ Conjecture} is the statement:
\begin{quote}
There is a proper class of regular cardinals
that are not $\omega$-strongly measurable in $\mathrm{HOD}$.
\end{quote}
\end{definition}

It turns out that if $\delta$ is an extendible cardinal,
then the HOD Conjecture is equivalent to the failure of clause (2)
of the dichotomy, Theorem~\ref{hoddichotomy},
which is the conjecture we discussed in the previous two sections. 
In particular, a model in which the HOD conjecture fails cannot be obtained by forcing.
It is clear that if $\mathrm{HOD}$ is correct about singular cardinals and computes their successors correctly  (clause (1) of Theorem \ref{hoddichotomy}) then 
the $\HOD$ Conjecture holds, as 
$$\{\gamma^+\mid\gamma\in \mbox{On and } \gamma \mbox{ is a singular cardinal}\}$$ is a proper class of regular cardinals which are not $\omega$-strongly measurable in $\mathrm{HOD}$.

We close this section with additional
remarks on the status of the HOD Conjecture.
\begin{itemize}
\item[(i)] It is not known whether more than 3 regular cardinals which are $\omega$-strongly measurable in HOD can exist.
\item[(ii)] Suppose $\gamma$ is a singular cardinal, $\cof(\gamma)>\omega$ and $\vert V_\gamma\vert=\gamma$. It is not known whether 
$\gamma^+$ can be $\omega$-strongly measurable in HOD.
\item[(iii)] Let $\delta$ be a supercompact cardinal. It is not known
whether any regular cardinal above $\delta$ can be $\omega$-strongly measurable in HOD.
\end{itemize}

\section{Supercompactness} \label{super}

Recall that
a cardinal $\delta$ is 
$\gamma$-supercompact iff there is a transitive class $M$
and an elementary embedding $j : V \to M$ such that $\mathrm{crit}(j) = \delta$,
$j(\delta) > \gamma$ and ${}^\gamma M \subseteq M$.
Also,
$\delta$ is a supercompact cardinal
iff $\delta$ is $\gamma$-supercompact cardinal for every $\gamma > \delta$.
If $\delta$ is an extendible cardinal,
then $\delta$ is supercompact
and
$\{ \alpha < \delta \mid \alpha \mbox{ is supercompact}\}$ is stationary in $\delta$ (see \cite{Kanamori} Theorem 23.7).

There is a standard first-order way to express  supercompactness in terms of
measures, which we review.
First suppose that $j : V \to M$ witnesses that $\delta$ is a $\gamma$-supercompact.
Observe that $j[\gamma] \in M$.  If we define
$$\mathcal{U} = \{ X \subseteq \mathcal{P}_\delta(\gamma) \mid j[\gamma] \in j(X) \},$$
then
$\mathcal{U}$ is a $\delta$-complete ultrafilter on $\mathcal{P}_\delta (\gamma)$.
Moreover,
$\mathcal{U}$ is {\em normal} in the sense that if $X \in \mathcal{U}$
and $f$ is a choice function for $X$,
then there exists $Y \in \mathcal{U}$ and $\alpha <  \gamma$
such that $Y \subseteq X$ and $f(\sigma) = \alpha$ for every $\sigma \in Y$.
Equivalently,
if $\langle X_\alpha \mid \alpha < \gamma \rangle$ is a sequence of sets from $\mathcal{U}$,
then the diagonal intersection,
$$\displaystyle\operatorname*{\Delta}_{\alpha < \gamma} \ X_\alpha = \{ \sigma \in \mathcal{P}_\delta (\gamma) \mid
\sigma \in X_\alpha \mbox{ for every } \alpha \in \sigma \}$$
also belongs to $\mathcal{U}$.
In addition,
$\mathcal{U}$ is {\em fine} in the sense that 
for every $\alpha < \gamma$,
$$\{ \sigma \in \mathcal{P}_\delta (\gamma) \mid \alpha \in \sigma \} \in \mathcal{U}.$$

Suppose, instead, that we are given
a $\delta$-complete ultrafilter $\mathcal{U}$ on $\mathcal{P}_\delta(\gamma)$
which is both normal and fine.  Then the ultrapower map derived from $\mathcal{U}$ can be shown to witness that $\delta$ is a $\gamma$-supercompact cardinal.
We might refer to such an ultrafilter (fine,  normal and $\delta$-complete) as a \textit{$\gamma$-supercompactness measure}.

Less well-known is the following characterisation of $\delta$ being supercompact
that more transparently related to extendibility.

\begin{theorem}[Magidor] \label{magidor}
A cardinal $\delta$ is supercompact iff for all  $\kappa > \delta$ and $a \in V_\kappa$,
there exist 
$\bar \delta < \bar \kappa < \delta$, $\bar a \in V_{\bar \kappa}$
and an elementary embedding
$ j: V_{{\bar \kappa} +1} \to V_{\kappa+1}$
such that
$\mathrm{crit}(j) = \bar \delta$, $j(\bar \delta ) = \delta$ and $j ( \bar a ) = a$. 
\end{theorem}

\begin{proof}
First we prove the forward direction.
Given $\kappa$ and $a$,
let $\gamma = |V_{\kappa + 1}|$
and $j : V \to M$ witness that $\delta$ is a $\gamma$-supercompact cardinal.
Then
$$j \upharpoonright V_{\kappa + 1} \in M$$
and witnesses the following sentence in $M$:
``There exist 
$\bar \delta < \bar \kappa < j(\delta)$, $\bar a \in V_{\bar \kappa}$
and an elementary embedding
$i: V_{{\bar \kappa} +1} \to V_{j(\kappa)+1}$
such that
$\mathrm{crit}(i) = \bar \delta$, $i(\bar \delta ) = j(\delta)$ and $i ( \bar a ) = j(a)$.''
Since $j$ is elementary, we are done.

For  the reverse direction,
let $\gamma > \delta$ be given.
Apply the right side with $\kappa = \gamma + \omega$.
(The choice of $a$ is irrelevant.)
This yields
$\bar \kappa $, $\bar \delta$ and $j$ as specified.
Take $\bar\gamma$ such that $j(\bar \gamma) = \gamma$.
Now $j[\bar{\gamma}] \in V_{\kappa+1}$ so it induces 
a normal fine ultrafilter $\bar{\mathcal{U}}$ on $\mathcal{P}_{\bar{\delta}}(\bar{\gamma})$.
Observe that $\bar{\mathcal{U}} \in V_{\bar{\gamma}+\omega}$, so we can define $\mathcal{U}=j(\bar{\mathcal{U}})$. Then, by elementarity, $V_{\kappa+1}$ believes that $\mathcal{U}$ is a normal fine ultrafilter on $\mathcal{P}_{\delta}(\gamma)$, and is large enough to bear true witness to such a belief. Thus $\delta$ is $\gamma$-supercompact.
\end{proof}

We will use the Solovay splitting theorem. A proof can be found in \cite{Jech} Theorem 8.10 or, using generic embeddings, in \cite{Jech} Lemma 22.27.

\begin{theorem}[Solovay]\label{solovaysplitting}
Let $\gamma$ be a regular uncountable cardinal.
Then every stationary subset of $\gamma$ 
can be partitioned into $\gamma$ many stationary sets.
\end{theorem}

We will also make key use of the 
following theorem, which provides a single set
that belongs to every $\gamma$-supercompactness measure on $\mathcal{P}_{\delta} (\gamma)$.
We will refer to this set as the {\em Solovay set}.

\begin{theorem}[Solovay] \label{solovayset}
Let $\delta$ be supercompact and $\gamma > \delta$ be regular. Then there exists an $X \subseteq \mathcal{P}_{\delta} (\gamma)$ such that the $\sup$ function is injective on $X$ and every $\gamma$-supercompactness measure contains $X$.
\end{theorem}

\begin{proof}
Let $\langle S_{\alpha} \,|\, \alpha < \gamma \rangle$ be a partition
of $\gamma \cap \cof(\omega)$ into stationary sets,
which exists by Theorem \ref{solovaysplitting}.
For $\beta < \gamma$ such that $\omega < \cf(\beta) < \delta$,
let $\sigma_\beta$ be the set of $\alpha < \beta$ such that $S_\alpha$ reflects to $\beta$.
In other words,
$$\sigma_\beta = \{\alpha < \beta \mid S_\alpha \cap \beta \mbox{ is stationary in } \beta \}.$$
Leave $\sigma_\beta$ undefined otherwise.
Note that it is not possible to partition $\beta$ into more that $\cf(\beta)$-many stationary sets, as can be seen by considering their restrictions to a club in $\beta$ of order type $\cf(\beta)$,
so, $\sigma_{\beta} \in \mathcal{P}_{\delta}(\gamma)$.
Define $X = \{\sigma_{\beta} \mid \sup(\sigma_{\beta})=\beta \}$.
Clearly,  the $\sup$ function is an injection on $X$ so given
$\mathcal{U}$ be a normal fine ultrafilter on 
$\mathcal{P}_\delta (\gamma)$ 
it remains to see that $X \in \mathcal{U}$. 
Let $j:V \to M$ be the embedding associated to $\mathcal{U}$.
In fact 
$\mathcal{U}$ is the corresponding ultrafilter derived from $j$,
so what we need to see is that
$$j[\gamma] \in j(X).$$ 
Let $\beta = \sup ( j[\gamma] )$
and
$$\langle S^*_\alpha \mid \alpha < j(\gamma) \rangle
= j ( \langle S_\alpha \mid \alpha < \gamma \rangle ).$$
Clearly,
$\beta < j(\gamma)$ and $\omega < \mathrm{cf}(\beta) < j(\delta)$,
so we are left to show that
$$j[\gamma]
=
\{ \alpha < \beta \mid
M \models 
S^*_\alpha \cap \beta \mbox{ is stationary in } \beta \}.$$

First we show containment  in the forward direction.
Consider any $\eta < \gamma$.
Then we want $S^*_{j(\eta)} = j(S_\eta)$ to be stationary.
Given $C$ a club subset of $\beta$ that belongs to $M$, 
define $D = \{ \alpha < \gamma \mid j(\alpha) \in C \}$.
Because $j$ is continuous at ordinals of countable cofinality,
$D$ is an $\omega$-club in $\gamma$.
But $S_\eta$ contains only ordinals of countable cofinality
and is stationary in $\gamma$ so
$S_\eta \cap D \not= \emptyset$.
Hence $j(S_\eta) \cap C \not= \emptyset$.

For containment in the reverse direction, consider any
$\alpha < \beta$ such that, in $M$,
$S^*_\alpha \cap \beta$
is stationary in $\beta$.
Working in $M$,
as $j[\gamma]$ is an $\omega$-club in $\beta$
and $S^*_\alpha$ contains only ordinals of countable cofinality,
there exists $\eta < \gamma$
such that
$j(\eta) \in S^*_\alpha$.
But $j[\gamma]$ is partitioned by the $j[S_\theta]$ for $\theta<\gamma$ so we can take $\theta<\gamma$ such that $j(\eta)\in j[S_\theta]\subseteq j(S_\theta)=S^*_{j(\theta)}$. This means $S^*_\alpha \cap S^*_{j(\theta)} \neq \emptyset$ so $\alpha = j(\theta)\in j[\gamma]$.
\end{proof}

\begin{remark}
The proof of Theorem \ref{solovayset} can be easily generalised to prove the following. Assume that $j:V\to M$ is a $\gamma$-supercompact embedding, where $\gamma$ is regular. Let $\kappa<\gamma$ be also regular,  $\beta=\sup j[\gamma]$ and $\tilde\beta=\sup j[\kappa]$. Then given a partition $\langle S_\alpha\mid \alpha<\kappa\rangle$ of $\cof(\omega)\cap \gamma$ into stationary sets, we have that 
$$j[\kappa]=\{\alpha\in \tilde \beta\mid S_\alpha^*\cap  \beta \mbox{ is stationary in } \beta \},$$ where $\langle S_\alpha^*\mid \alpha< j(\kappa)\rangle= j(\langle S_\alpha \mid \alpha < \kappa\rangle)$.
\end{remark}

\section{Weak Extender Models} \label{wem}

In inner model theory,
the word {\em extender} has taken on a very general meaning
as any object that captures the essence of a given large cardinal property.
Sometimes ultrafilters or systems of ultrafilters are used.
At other times, elementary embeddings or restrictions of elementary embeddings
are more relevant.
We have already seen two first-order ways to express supercompactness.
An easier example is measurability:
if $U$ is a normal measure on $\kappa$ and $j : V \to M$
is the corresponding ultrapower map,
then $U$ and $j \upharpoonright V_{\kappa+1}$
carry exactly the same information.

Building a canonical inner model with a supercompact cardinal has been a major open
problem in set theory for decades.
Canonical inner models for measurable cardinals were produced early on.  Letting
$U$ be a normal measure on $\kappa$ and setting $\bar U = U \cap L[U]$,
we can see that $\bar U \in L[U]$
and $L[U] \models \mbox{$\bar U$ is a normal measure on $\kappa$}$.
The general theory of $L[U]$ does not depend on there being measurable cardinals in $V$
but this was an important first step.

\begin{definition}\label{wemfdsc}
A transitive class $N$ model of ZFC is called a {\em weak extender model for $\delta$ supercompact} iff for every $\gamma>\delta$ there exists a normal fine measure $\mathcal{U}$ 
on $\mathcal{P}_\delta(\gamma)$ such that
\begin{enumerate}
\item $N\cap \mathcal{P}_\delta(\gamma)\in \mathcal{U}$ and
\item $\mathcal{U} \cap N \in N$.
\end{enumerate}
\end{definition}

The first condition says that $\mathcal{U}$ {\em concentrates} on $N$.
In the case of the measurable cardinal,
which we discussed above,
we get the analogous first condition for free because $L[U] \cap \kappa = \kappa \in U$.
We might refer to the second condition
as saying that $\mathcal{U}$ is {\em amenable} to $N$.

\begin{lemma} \label{deltacoverwem}If $N$ is a weak extender model for $\delta$ supercompact,  then it has the $\delta$-covering property. 
\end{lemma}
\begin{proof}
Note that it is enough to prove $\delta$-covering for sets of ordinals. Now, given $\tau\subseteq\gamma$ with $\vert \tau\vert<\delta$,
let $\mathcal{U}$ be a $\gamma$-supercompactness  measure such that $N\cap \mathcal{P}_\delta(\gamma)\in \mathcal{U}$ and $\mathcal{U}\cap N\in N$. By fineness,
for each $\alpha<\gamma$, we have that $\{\sigma\in \mathcal{P}_\delta(\gamma)\mid \alpha\in\sigma\}\in \mathcal{U}$.
Hence, as $\vert\tau\vert<\delta$, by $\delta$-completeness we have $\{\sigma \in \mathcal{P}_\delta(\gamma)\mid \tau\subseteq \sigma\}$ belongs to $\mathcal{U}$. 
Also as $N\cap \mathcal{P}_\delta(\gamma)\in \mathcal{U}$,
there is a $\sigma\in N\cap \mathcal{P}_\delta(\gamma)$ and $\sigma\supseteq \tau$ as desired. 

\end{proof}

\begin{lemma}\label{wemclose}
Suppose $N$ is a weak extender model for $\delta$ supercompact  and $\gamma>\delta$ is such that $N\models `\gamma \mbox{ is a regular cardinal' }$. Then $\vert \gamma\vert = \cf(\gamma)$.
\end{lemma}
\begin{proof}
Let $\gamma >\delta$. Of course, $\cf(\gamma)\leq\vert\gamma\vert$. Now we prove the reverse inequality. By Lemma \ref{deltacoverwem}, $N$ satisfies the $\delta$-covering property so, as $N$ believes $\gamma$ is a regular cardinal, we have that $\cf(\gamma)\geq\delta$. Now fix $\mathcal{U}$ a $\gamma$-supercompactness measure, such that $N\cap \mathcal{P}_\delta(\gamma)\in \mathcal{U}$ and $\mathcal{U}\cap N\in N$. As $\gamma$ is a regular cardinal of $N$, we may apply Theorem \ref{solovayset} within $N$ and get  a Solovay set $X\in N$. So the $\sup$ function is an injection on $X$ and $X$ belongs to $\mathcal{U}$. Now fix  a club $D\subseteq\gamma$ of order type $\cf(\gamma)$ and define $A=\{\sigma\in \mathcal{P}_\delta(\gamma)\mid \sup(\sigma)\in D\}$. 

We first claim that $A\in\mathcal{U}$. Letting $j:V\to M$ be the ultrapower map induced by $\mathcal{U}$, it is enough to show that $j[\gamma]\in j(A)$.  
Define $\beta=\sup j[\gamma]$. By the definition of $A$, 
we need to see that $\beta\in j(D)$. Note that $j(D)$ is a club in $j(\gamma)$, and as $D$ is unbounded in $\gamma$
we have that $j[\gamma]\cap j(D)$ is unbounded in $\beta$.
Thus $j(D)$ being closed implies $\beta\in j(D)$. Hence $\{\sigma\in X\mid \sup(\sigma)\in D\}\in \mathcal{U}$. Recall that $\mathcal{U}$ is fine, so
$$\gamma=\bigcup \{\sigma\in X\vert \sup(\sigma)\in D\}.$$

Now, because the $\sup$ function is injective on $X$, we have that the cardinality of $\gamma$ is at most $\delta\,\vert D\vert$. But the order type of $D$ is $\cf(\gamma)$, so $\vert \gamma\vert\leq\delta\,\cf(\gamma)$. Finally remember $\delta\leq\cf(\gamma)$, so $\vert \gamma\vert\leq \cf(\gamma)$ which concludes the proof.
\end{proof}

\begin{corollary}\label{wemwc}
Let $N$ be a weak extender model for $\delta$ supercompact  and $\gamma>\delta$ be a singular cardinal, then
\begin{enumerate}
\item $N \models `\gamma\mbox{ is singular'}$ and
\item $\gamma ^{+}=(\gamma^{+})^N$.
\end{enumerate}
\end{corollary}
\begin{proof}
Immediate by Lemma \ref{wemclose}. 
\end{proof}

Next, we characterise the HOD Conjecture in two ways, each of  which says HOD is close to $V$ in a certain sense.

\begin{theorem}\label{equivalences}
Let $\delta$ be an extendible cardinal. The following are equivalent.
\begin{enumerate}
\item The $\mathrm{HOD}$ Conjecture.
\item $\mathrm{HOD}$ is a weak extender model for $\delta$ supercompact.
\item Every singular cardinal $\gamma>\delta$, is singular in $\mathrm{HOD}$ and $\gamma^+=(\gamma^+)^{\HOD}$.
\end{enumerate}
\end{theorem}

\begin{proof}
(2) implies (3) is just Lemma \ref{wemwc}. That (3) implies (1) was shown in the discussion right after the definition of the HOD Conjecture (Definition \ref{HODConjecture}). We now prove (1) implies (2).

Given $\zeta>\delta$, we wish to show that there is a $\zeta$-supercompactness measure $\mathcal{U}$ such that $\mathcal{U}\cap \HOD \in \HOD$ and $\mathcal{P}_\delta(\gamma)\cap \HOD\in \mathcal{U}$. For this, take $\gamma>2^\zeta$, such that $\vert V_\gamma\vert^{\mathrm{HOD}}=\gamma$ and fix a regular cardinal $\lambda>2^\gamma$ such that $\lambda$ is not $\omega$-strongly measurable in HOD. Finally, pick $\eta>\lambda$ such that  the defining formula for HOD is absolute for $V_\eta$, whence ${\mathrm{HOD}}^{V_\eta}=$HOD$\cap V_\eta$. As $\delta$ is extendible, there is an elementary embedding $j:V_{\eta+1}\to V_{j(\eta)+1}$ with critical point $\delta$.

\begin{claim}
 $j[\gamma]\in{\mathrm{HOD}}^{V_{j(\eta)}}.$ 
 \end{claim}

As $\lambda$ is not $\omega$-strongly measurable in $\mathrm{HOD}$ and $2^\gamma<\lambda$ (in $V$ and so in HOD) there is a partition $\langle S_\alpha\mid \alpha\in \gamma\rangle$ of $\cof(\omega)\cap\lambda$  into stationary sets such that $\langle S_\alpha\ \mid \alpha< \gamma\rangle \in \mathrm{HOD}$. Thus $\langle S_\alpha\mid\alpha\in\gamma\rangle\in{\mathrm{HOD}}^{V_{\eta}}$. By the elementarity of $j$ we have
$$ \langle S^*_\alpha\mid\alpha\in j(\gamma)\rangle=j\left(\langle S_\alpha\mid\alpha\in\gamma\rangle\right)\in{\mathrm{HOD}}^{V_{j(\eta)}}.$$
Let $\beta=\sup j[\lambda]$ and $\tilde \beta= \sup j[\gamma]$. By  the remark after the proof of Theorem \ref{solovayset}, 
$$j[\gamma]=\{\alpha\in\tilde\beta\,\vert\ S^*_{\alpha}\cap\beta\mbox{ is stationary in }\beta\}$$
This shows that $j[\gamma]$ is OD in $V_{j(\eta)}$. Moreover $V_{j(\eta)}$ is correct about stationarity in $\beta$, thus $j[\gamma]\in{\mathrm{HOD}}^{V_{j(\eta)}}$. Also note that $j[\zeta]\in{\mathrm{HOD}}^{V_{j(\eta)}}$.

Now, observe that $\HOD^{V_{j(\eta)}}\subset \HOD$, 
so we have that $j[\gamma]\in \HOD$. Also $\vert V_\gamma\vert^{\HOD}=\gamma$, so we may take $e\in$HOD a bijection from $\gamma$ to $V^{\mathrm{HOD}}_\gamma$. Clearly $j(e)[j[\gamma]]=j[V_\gamma\cap \mathrm{HOD}]$ and so $j[V_\gamma\cap \mathrm{HOD}]\in \HOD$. Furthermore, as $$j\upharpoonright (V_\gamma\cap \mathrm{HOD})$$ is the inverse of the Mostowski collapse, we have that
$$j\upharpoonright (V_\gamma\cap \mathrm{HOD})\in\mathrm{HOD}.$$
Now, let $\mathcal{U}$ be the ultrafilter on $\mathcal{P}_\delta(\zeta)$ derived from $j$. That is, for $A\subseteq \mathcal{P}_\delta(\zeta)$, $A\in \mathcal{U}$ iff $j[\zeta]\in j(A)$. So,
$$\mathcal{P}_\delta (\zeta) \cap\mathrm{HOD}\in\mathcal{U} \mbox{ as } j[\zeta]\in {\mathrm{HOD}}^{V_{j(\eta)}}= j(\mathrm{HOD}\cap V_\eta)$$
$$\mathcal{U}\cap\mathrm{HOD}\in \mbox{HOD as } j\upharpoonright (V_\gamma\cap \mathrm{HOD})\in\mathrm{HOD} \mbox{ and }\gamma> 2^{\zeta}.$$
Thus $\mathcal{U}$ concentrates on HOD and is amenable to HOD as desired.
\end{proof}

As a corollary, we obtain the following version of the HOD Dichotomy, Theorem \ref{hoddichotomy}.

\begin{corollary}\label{final}
Let $\delta$ be an extendible cardinal. Then exactly one of then following holds.
\begin{enumerate}
\item For every singular cardinal $\gamma>\delta$, $\gamma$ is singular in $\mathrm{HOD}$ and\\
 $\gamma^+=(\gamma^+)^{\HOD}$.  
\item There exists a $\kappa>\delta$ such that every regular $\gamma>\kappa$ is measurable in $\mathrm{HOD}$.
\end{enumerate}
\end{corollary}

\begin{proof}
Suppose (2) does not hold, then there are arbitrarily large regular cardinals that are not measurable in HOD. By Lemma \ref{omega} there are arbitrarily  large regular cardinals that are not $\omega$-strongly measurable in HOD. Now by the proof of Theorem \ref{equivalences}, this implies that $\mathrm{HOD}$ is a weak extender model for $\delta$ supercompact. Finally Corollary \ref{wemwc} yields (1).
\end{proof}

\section{Elementary Embeddings of weak extender models} \label{embed}

We now give more evidence that if $N$ is weak extender model for $\delta$ supercompact  then it is close to $V$.
We will prove that if $\delta$ an extendible cardinal, $N$ is a weak extender model for $\delta$ supercompact and $j$ is an elementary embedding between levels of $N$  with $\crit(j)\geq \delta$, then $j\in N$. 
This implies that if $\delta$ is extendible and the $\mathrm{HOD}$ Conjecture holds then there are no elementary embeddings from $\mathrm{HOD}$ to $\mathrm{HOD}$ with critical point greater or equal $\delta$. This says that a natural analog of $0^\#$ for HOD does not exist.
As one would expect from Magidor's characterisation of supercompactness, Theorem \ref{magidor}, there is an alternative formulation of ``weak extender model for $\delta$ supercompact'' in terms of suitable elementary embeddings $j:V_{\bar\kappa+1}\to V_{\kappa+1}$ for $\bar\kappa <\delta$.

\begin{theorem}\label{mcowem}
Let $N$ be a proper class model of ZFC. Then the following are equivalent:
\begin{enumerate}
\item $N$ is a weak extender model for $\delta$ supercompact.
\item For every $\kappa>\delta$ and $b\in V_\kappa$, there exist two cardinals $\bar{\kappa}$ and $\bar{\delta}$ below $\delta$, $\bar{b}\in V_{\bar\kappa}$ and $j:V_{\bar{\kappa}+1}\to V_{\kappa+1}$ such that: 
$$\mbox{crit}(j)=\bar{\delta},\mbox{   } j(\bar{\delta})=\delta, \mbox{    } j(\bar{b})=b,$$
$$j(N\cap V_{\bar{\kappa}})=N\cap V_{\kappa}\mbox{ and}$$
$$j\upharpoonright (V_{\bar{\kappa}}\cap N)\in N.$$
\end{enumerate}
\end{theorem}

\begin{proof}[Proof (2) implies (1)]
Given $\gamma>\delta$,  we may assume $\gamma=\vert V_{\gamma}\vert$. Let $\bar{\kappa}=\gamma+\omega$.  We obtain $\bar{\kappa}$, $\bar{\delta}$ and $j$ using (2). Take $\bar\gamma$ such that $\bar{\kappa}=\bar{\gamma}+\omega$, whence $j(\bar\gamma)=\gamma$. Let $\bar{\mathcal{U}}$ be the measure on $\mathcal{P}_{\bar \delta}(\bar\gamma)$ derived form $j$. That is, for $A\in \mathcal{P}_{\bar\delta}(\bar\gamma)$
$$A\in\bar{\mathcal{U} }\iff j[\bar\gamma] \in j(A).$$
Define $\mathcal{U}=j(\bar{\mathcal{U}})$. We show that $\mathcal{U}$ is a $\gamma$-supercompactness measure such that $\mathcal{P}_{\delta}(\gamma)\cap N\in \mathcal{U}$ and $\mathcal{U}\cap N  \in N$. 

We claim that $P_{\bar\delta}(\bar\gamma)\cap N\in \bar{\mathcal{U}}$. By (2), we know that $ j(N\cap V_{\bar\kappa})=N\cap V_\kappa$. Thus for every $a\in V_{\bar\kappa}$ we have 
$$j(a\cap N)= j(a)\cap j(N\cap V_{\bar\kappa})=j(a)\cap N.$$
Recalling that $\bar \kappa=\bar \gamma+\omega$,
$$j\left(P_{\bar\delta}(\bar\gamma)\cap N\right)= \mathcal{P}_\delta(\gamma)\cap N.$$
Now, as $j\upharpoonright(N\cap V_{\bar\kappa})\in N$,
we have $j[\bar\gamma]\in N$, so $j[\bar\gamma]\in \mathcal{P}_\delta(\gamma)\cap N=j (P_{\bar\delta}(\bar\gamma)\cap N)$, which readily implies our claim.

Finally, by elementarity of $j$, we have that $\mathcal{U}$ is a fine and normal measure on $\mathcal{P}_\delta(\gamma)$ and, by the previous claim, $j(N\cap P_{\bar\delta}(\bar\gamma))\in \mathcal{U}$. It follows that  
$$N\cap \mathcal{P}_\delta(\gamma)\in \mathcal{U}.$$ 
Moreover, as $j\upharpoonright (N\cap V_{\bar\kappa})\in N$, we have that $j(\bar{\mathcal{U}}\cap N)\in N$.  Hence $$j(\bar{\mathcal{U}}\cap N)=j(\bar{\mathcal{U}})\cap N= \mathcal{U}\cap N \in N.$$
This concludes the first direction.
\end{proof}

\begin{proof}[Proof (1) implies (2)] Let $\kappa>\delta$, $b\in V_\kappa$ and fix $\gamma>\vert V_{\kappa+\omega} \vert$ such that 
$\vert V_\gamma\vert^N=\gamma$. 
Fix a $\gamma$-supercompactness measure $\mathcal{U}$ such that $\mathcal{P}_{\delta}(\gamma)\cap N\in \mathcal{U}$ and $\mathcal{U}\cap N\in N$. 
Now, fix  a bijection $e:\gamma\to V_\gamma^N$ in $N$. We now work in $N$.
Define $N_\sigma$  to be the Mostowski collapse of  $e[\sigma]$, and 
$$Y=\lbrace\sigma \in N\cap \mathcal{P}_\delta(\gamma)\mid N_\sigma = V^N_{\mbox{otp}(\sigma)}\rbrace.$$ 
Hence $Y$ is a club of $N\cap \mathcal{P}_\delta(\gamma)$, so it belongs to $\mathcal{U}\cap N$.
Thus if $j:V\to M$ is the ultrapower map it follows that $j[\gamma]\in j(Y)$. This implies from the definition of $Y$ that the collapse of $j(e)[j[\gamma]]$ is exactly $V_\gamma \cap j(N\cap V_\delta)$. Note also that $j(e)[j[\gamma]]=j[V_\gamma \cap N]$.
Of course $V_\gamma\cap N$ is the collapse of $j[V_\gamma \cap N]$, so $$V_\gamma \cap N= V_\gamma \cap j(N\cap V_\delta),$$ which implies 
$$ V_{\kappa}\cap N= V_{\kappa}\cap j(N\cap V_\delta).$$

It is clear that for $\sigma \in N$ we have that $e[\sigma]\in V_{\gamma +1}\cap N$. Then by \L os' Theorem we have that $j[V_\gamma \cap N]=j(e)[j[\gamma]]\in j(V_{\gamma +1}\cap N)$. Notice that the collapsing map of $ j[V_\gamma \cap N]$ is just the inverse of $j\upharpoonright (V_\gamma \cap N)$, thus
$$j\upharpoonright (V_{\kappa} \cap N)\in j(V_{\gamma +1}\cap N).$$
Now $M$ being closed under $\gamma$ sequences and $\gamma>\vert V_{\kappa+\omega} \vert$ imply $j\upharpoonright(V_{\kappa+1})$ belongs to  $M$. 
Working in $M$ let $i=j\upharpoonright(V_{\kappa+1})$. 
Now let us prove that the two previous equations imply that $i$ satisfy the conditions of \em(2) \em relative to $j(\kappa)$, $j(b)$ and $j(N\cap V_{\gamma+1})$ in $M$. Indeed the equations give
$$i\upharpoonright\left(V_{\kappa}\cap j(N\cap V_{\gamma+1})\right)=i\upharpoonright \left(V_{\kappa} \cap N\right)\in N\cap j(V_{\gamma+1}).$$
Furthermore as $i$ and $j$ agree, 
$$i(V_{\kappa}\cap j(N\cap V_{\gamma+1}))=i(V_\kappa\cap N)=j(N\cap V_{\gamma+1})\cap V_{j(\kappa)}.$$
Also $j(b)=i(b)$, so by elementarity (2) holds in $V$ with respect to $\kappa$, $b$ and $N$.

\end{proof}

We now prove that if $\delta$ is an extendible cardinal and $N$ is a weak extender model for $\delta$ supercompact, then $N$ sees all elementary embeddings between its levels. 

\begin{theorem}\label{elemofN}
Let $\delta$ be an extendible cardinal. Assume that $N$ is a weak extender model for $\delta$ supercompact and $\gamma>\delta$ is a cardinal in $N$. Let 
$$ j: H(\gamma^+)^N\to H(j(\gamma)^+)^N$$ 
be an elementary embedding with $\delta\leq\mathrm{crit}(j)$ and $j\neq \mathrm{id}$. Then $j\in N$.
\end{theorem}

\begin{proof}
Define $b=(j,\gamma)$ and let $\kappa$ be a cardinal much larger than $j(\gamma)$. Now, as $N$ is a weak extender model for $\delta$ supercompact, we may apply Theorem \ref{mcowem} to $\kappa$ and $b$. Hence, we get  an elementary embedding $\pi: V_{\bar\kappa+1}\to V_{\kappa+1}$, two ordinals $\bar\delta$, $\bar\gamma$ and $\bar{\jmath}\in V_{\bar \kappa}$, with the following properties
$$j(N\cap V_{\bar\kappa})=N\cap V_\kappa,\;\pi\upharpoonright\left(V_{\bar\kappa}\cap N\right)\in N$$ 
and 
$$\mathrm{crit}(\pi)=\bar\delta, \;\pi(\bar{\jmath})=j,\; \pi(\bar\delta)=\delta,\;\pi(\bar\gamma)=\gamma, \; \bar\kappa<\delta .$$
Hence, by the elementarity of $\pi$, we have that $\bar{j}: H(\bar{\gamma}^+)^N\to H(\bar{\jmath}(\bar\gamma)^+)^N$ is an elementary map with $\bar\delta\leq\mbox{crit}(\bar{\jmath})$. 
Furthermore as $\bar\kappa$ is very large above $\bar{\jmath}(\bar\gamma)$, we have that 
$$\pi\upharpoonright\left(H(\bar{\jmath}(\bar\gamma)^+)^N\right)\in N,$$ 
hence $ \pi\upharpoonright\left(H(\bar{\jmath}(\bar\gamma)^+)^N\right)\in H(\gamma^+)^N.$
Define $\pi^*=j\left(\pi\upharpoonright \left(H(\bar{\jmath}(\bar\gamma)^+)^N\right)\right)$. Now we wish to show that $\bar{\jmath}\in N$. This will be done by proving that $N$ can actually compute $\bar \jmath$. For this, take $\bar{a}\in H(\bar\gamma^+)^N$ and $\bar{s}\in H(\bar{\jmath}(\bar\gamma)^+)^N$. Let $\pi(\bar a)=a$ and $\pi(\bar s)=s$ then, 
\begin{align*}
\bar s\in\bar \jmath (\bar a) &
\iff s\in j(a)\\ 
&\iff s\in j(\pi(\bar a))\\
&\iff \pi(\bar s)\in j(\pi\upharpoonright\left(H(\bar{\jmath}(\bar\gamma)^+)^N\right)(\bar a))\\
&\iff \pi(\bar s)\in\pi^*(j(\bar a))\\
&\iff \pi\upharpoonright\left(H(\bar{\jmath}(\bar\gamma)^+)^N\right)(\bar s)\in\pi^*(\bar a).
\end{align*}
Where the last equivalence follows because $\mathrm{crit}(j)>\bar \kappa$ and $\bar \kappa$ is sufficiently large above $\bar\jmath (\bar\gamma^+)^N$.
Now as $\pi^*$ and $\pi\upharpoonright\left(H(\bar{\jmath}(\bar\gamma)^+)^N\right)$ are in $N$, $\bar \jmath \in N\cap V_{\bar\kappa}$. Since $\pi$ stretches $N$ correctly up to rank $\bar\kappa$, we conclude $j=\pi(\bar \jmath)\in N$ as desired.
\end{proof}

Now we show that, if $\delta$ is an extendible cardinal, then no elementary embedding maps a weak extender model for $\delta$ supercompact to itself. For this we recall the following form of Kunen's theorem.

\begin{theorem}[Kunen]\label{kun}
Let $\kappa$ be an ordinal. Then there is no non-trivial elementary embedding
 $$i:V_{\kappa+2}\to V_{\kappa+2}.$$
\end{theorem}

The proof can be found in \cite{Kanamori} Theorem 23.14.

\begin{theorem}\label{noj}
Let $N$  be a weak extender model for $\delta$ supercompact. Then there is no elementary embedding $j:N\to N$ with $\delta\leq\mathrm{crit}(j)$ and $j\neq\mathrm{id}$.
\end{theorem}

\begin{proof}
Suppose for contradiction that there is such a $j$. Let $\kappa>\delta$ be a fixed point of $j$. Then the restriction of $j$ to $V_{\kappa+2}^N$ is the an elementary embedding  $i:V^N_{\kappa+2}\to V^N_{\kappa+2}$ with crit$(i)\geq \delta$. Theorem \ref{elemofN} implies $i\in N$. This contradicts Theorem \ref{kun} within $N$.
\end{proof}

\begin{corollary}
Assume the $\mathrm{HOD}$ Conjecture. If $\delta$ is an extendible cardinal, then there is no
$j:\mathrm{HOD}\to \mathrm{HOD} \mbox{ with }\delta\leq{crit}(j) \mbox{ and }j\neq\mathrm{id}.$
\end{corollary}

\begin{proof}
Follows from the Theorem \ref{noj} and Theorem \ref{equivalences}. 
\end{proof}

Finally we give and example $N$ of a weak extender model for $\delta$ supercompact  other than $V$. $N$ will be such that there is  and nontrivial elementary embedding $j:N\to N$, with $\crit(j)<\delta$. The point of the next example is that actually one can have a weak extender model for $\delta$ supercompact but it lacks structural properties, such as the ones $\mathrm{HOD}$ and $L$ possess. Note that this makes Theorem \ref{noj} actually optimal.

For the example we will use the following fact.

\begin{lemma}\label{restr}
Let $\kappa$ be a measurable cardinal, $\mu$ a measure on $\kappa$ and $j:V\to M$ the ultrapower map given by $\mu$. Also let $\nu$ be a $\delta$-complete measure, for some $\delta>\kappa$, $k: V\to N$ the ultrapower map given by $\nu$ and $l:M\to \mathrm{Ult}(M,j(\nu))$ the ultrapower map. Then $k\upharpoonright M=l$.
\end{lemma}

\begin{proof}
First, observe that as the critical point of $k$ is above $\kappa$, then $\mu\in N$. Apply $\mu$ to $N$ and let $j^{\prime}:N\to \mathrm{Ult}(N,\mu)$ be the ultrapower map. $^{\delta}N\subset N$ because $\nu$ is $\delta$-complete, so all functions from $\kappa$ to $N$ are in $N$, and this readily implies $j\upharpoonright N=j^{\prime}$.

Now, for $j(f)(\kappa)$ an element of $M$, we wish to see that $k(j(f)(\kappa))=l(j(f)(\kappa))$. For simplicity, $j$ for a restriction of $j$ to a suitable rank-initial segment which can then be treated as an element; likewise for $k$. By elementarity we have that $k(j(f)(\kappa)) = k(j)(k(f))(k(\kappa))$, but $k(j)$ is $j^{\prime}$ which is the restriction of $j$ to $N$ so,
\begin{align*}
    k(j(f)(\kappa)) &= j(k(f))(\kappa)\\
        &= j(k)(j(f))(\kappa)\\
        &= l(j(f))(\kappa)\\
        &= l(j(f))(l(\kappa))\\
        &= l(j(f)(\kappa))
\end{align*}
In other words, $k$ restricts to $l$ as desired.
\end{proof}

\begin{example}
Let $\delta$ be a supercompact cardinal. Then there is $N$ a weak extender model for $\delta$ supercompact , and a nontrivial $j:N\to N$ with $\crit(j)<\delta$.
\end{example}
Let $\kappa<\delta$ be a measurable cardinal and take $\mu$ a measure on $\kappa$. Let
$$V=M_0\to M_1\to M_2\to M_3 \to \cdots \to M_\omega$$ 
be the internal iteration of $V$ by $\mu$ of length $\omega$. So we have $M_0=V$, $\kappa_0=\kappa$; and inductively for naturals $n>0$ define  $\mu_n=i_{n-1,n}(\mu_{n-1})$, $\kappa_n=i_{n-1}(\kappa_{n-1})=$crit$(\mu_n)$ and let $i_{n,n+1}: M_n\to M_{n+1}$  be the map induced by taking the ultrapower of $M_n$ by $\mu_n$. $M_\omega$ is then the direct limit of the system and $i_{n,\omega}:M_n\to M_\omega$ the induced embeddings. $M_\omega$ is well founded and so we identify it with its transitive collapse (see Theorem 19.7 of  \cite{Jech} ).
Define $N=M_\omega$. 

Now, we show that $N$ is a weak extender model for $\delta$ supercompact. This is equivalent to showing that for unboundedly many $\gamma$ there is a $\gamma$-supercompactness measure that concentrates on $N$ and is amenable to $N$. Note that $i_{0,\omega}(\delta)=\delta$ and that for unboundedly many ordinals $\gamma$, we have that $i_{0,\omega}(\gamma)=\gamma$.  Fix such $\gamma$ and let $\mathcal{U}$ be a normal and fine measure on $\mathcal{P}_\delta(\gamma)$. We prove that $\mathcal{U}$ is a suitable measure for $N$.
Now, let $\mathcal{W}=i_{0,\omega}(\mathcal{U})$, and $\mathcal{W}_n=i_{0,n}(\mathcal{U})$ (observe that, for each $n$, $\mathcal{W}_n$ is a normal fine measure on $\mathcal{P}_\delta(\gamma)$ in $M_n$). By Lemma \ref{restr} the map induced by taking $\mathrm{Ult}(V, \mathcal{U})$ restricts to the one given by $\mathrm{Ult}(M_1, \mathcal{W})$. Inductively we have that if $k_n:M_n\to \mbox{Ult}(V,\mathcal{W}_n)$ is the ultrapower map, then $k_n=k_0\upharpoonright M_n$. If follows then that as $i_{n,\omega}(\mathcal{W}_n)=\mathcal{W}$ (for each $n$) and $N$ is the direct limit of the initial system, we have that $k=k_0\upharpoonright N$, where $k:N\to\mbox{Ult}(N,\mathcal{W})$ is the ultrapower map given by $\mathcal{W}$. Therefore for $A\in N$, $k[\gamma]\in k(A)$ iff $k_0[\gamma]=k_0(A)$, which readily implies $\mathcal{W}=\mathcal{U}\cap N$; in other words $\mathcal{U}$ is amenable to $N$.  Also, $\mathcal{U}$ concentrates on $N$ as $N\cap \mathcal{P}_\delta(\gamma)\in \mathcal{W}\subseteq \mathcal{U}$, as desired. Thus $N$ is a weak extender model for $\delta$ supercompact.

Finally, observe that if $j=i_{0,1}\upharpoonright N$ then $j:N\to N$ as $N$ is the $\omega$-th iterate, so we have a nontrivial embedding from $N$ to $N$, the key point here is that $\crit(j)<\delta$.

\section{Consequences of the HOD Conjecture}\label{consequences}

We conclude by summarising  without proof some results that would follow if the HOD Conjecture were proved to be a theorem of ZFC.

\begin{theorem} [ZF] \label{ACclose}
Assume that ZFC proves the HOD Conjecture. Suppose $\delta$ is an extendible cardinal.
Then there is a transitive class $M \subseteq V$ such that:
\begin{enumerate}
	\item $M\models \ZFC$
	\item $M$ is $\Sigma_2(a)$-definable for some $a \in V_{\delta}$
	\item Every set of ordinals is $<\delta$-generic over $M$
	\item $M \models$ ``$\delta$ is an extendible cardinal"
\end{enumerate}
\end{theorem}

The conclusion of the theorem is that there is an inner model M which is both close to V and in which the Axiom of Choice holds. (See \cite{Woodin} Theorem 229 for proof of a stronger result.) This is close to ``proving'' the Axiom of Choice from large cardinal axioms and suggests the following conjecture. 

\begin{definition}
The {\em Axiom of Choice Conjecture} asserts in ZF, that if $\delta$ is an extendible cardinal then the Axiom of Choice holds in $V[G]$, where $G$ is $V$-generic for collapsing $V_{\delta}$ to be countable.
\end{definition}

One application of Theorem \ref{ACclose} is the following theorem. (See \cite{Woodin} Theorem 228 for a proof.)

\begin{theorem}[ZF]
Assume that $\ZFC$ proves the $\HOD$ Conjecture. Suppose $\delta$ is an extendible cardinal. Then for all $\lambda > \delta$ there is no non-trivial elementary embedding $j:V_{\lambda+2} \to V_{\lambda+2}$.
\end{theorem}

Thus (assuming that $\ZFC$ proves the $\HOD$ Conjecture) one nearly has a proof of Kunen's Theorem (\ref{kun}) without using the Axiom of Choice. 

For our final theorem we need a new definition. $L(\mathcal{P}(OR))$ is built in the same way as the usual $L$-hierarchy but allowing the use of all sets of ordinals in definitions. So it is the least model of ZF that contains all sets of ordinals. Note that, under ZF, this is not necessarily the whole of $V$.

\begin{theorem}[ZF]
Assume that ZFC proves the HOD Conjecture. Suppose that $\delta$ is an extendible cardinal.
Then in $L(\mathcal{P}(OR))$:
\begin{enumerate}
	\item $\delta$ is an extendible cardinal.
	\item The Axiom of Choice Conjecture holds.
\end{enumerate}
\end{theorem}

\bibliographystyle{plain}

\end{document}